\documentclass [a4paper, 11pt]{article}
\usepackage{amsthm,amsmath,amssymb}
\usepackage{mathrsfs}
\usepackage{enumerate}
\usepackage{amsfonts}
\usepackage{verbatim}
\usepackage{amsbsy}
\usepackage{amsmath}
\usepackage{amssymb}
\usepackage{MnSymbol}
\usepackage{mathrsfs}
\usepackage{graphicx}
\usepackage{caption}
\usepackage{cite}
\usepackage{array}
\usepackage{stackengine}
\usepackage[table,xcdraw]{xcolor}
\usepackage{colortbl}
\usepackage{multirow}
\usepackage{booktabs,xltabular}
\usepackage{graphicx}
\usepackage{chemfig}
\newcolumntype{P}[1]{>{\centering\arraybackslash}p{#1}}
\newcolumntype{M}[1]{>{\centering\arraybackslash}m{#1}}
\usepackage{soul}

\setchemfig{atom sep=15pt}

\parskip \bigskipamount

\DeclareMathOperator{\sig}{sig}
\DeclareMathOperator{\dist}{dist}
\DeclareMathOperator{\rt}{root}
\DeclareMathOperator{\diam}{diam}
\usepackage[symbol]{footmisc}

\title {Extremal Trees With Prescribed Burning Numbers} 
\author{\stepcounter{footnote}Eugene Jun Tong Leong \thanks{School of Mathematical Sciences, Universiti Sains Malaysia, 11800 USM, Malaysia. Email: EugeneLeong@student.usm.my} \and Kai An Sim\thanks{School of Mathematical Sciences, Sunway University, 47500 Bandar Sunway, Malaysia. Email: kaians@sunway.edu.my} \and Wen Chean Teh \!$^*$ \thanks{School of Mathematical Sciences, Universiti Sains Malaysia, 11800 USM, Malaysia. Email: dasmenteh@usm.my} }
\date{}

\theoremstyle{plain}
\newtheorem{theorem}{Theorem}[section]

\newtheorem{lemma}[theorem]{Lemma}
\newtheorem{corollary}[theorem]{Corollary}

\newtheorem*{burningconjecture}{Burning number conjecture}
\theoremstyle{definition}
\newtheorem{definition}[theorem]{Definition}

\theoremstyle{remark}
\newtheorem{remark}[theorem]{Remark}

\begin{document}
	\maketitle
	\begin{abstract}
		Graph burning is motivated by the spread of social influence, and the burning number measures the speed of the spread. Given that the smallest burning number among the spanning trees of a graph determines the burning number of a connected graph, trees are the main objects of investigation in graph burning.
Given a prescribed burning number, our study focuses on identifying the corresponding extremal trees with respect to order up to graph homeomorphism. In this work, we propose the concept of admissible sequences over a homeomorphically irreducible tree in addition to developing a general framework. We then determine whether an admissible sequence induces an extremal tree with a specified burning number. Additionally, we obtain some results on the smallest attainable diameter for extremal $n$-spiders with a prescribed burning number. 		
	\end{abstract}

	\footnotesize $^*$Corresponding author
	\normalsize
	
	\section{Introduction}
	
	Bonato et al.~\cite{bonato2016how} proposed graph burning to measure how quickly a social influence spreads in a social network. The concept of graph burning has similarities to the firefighter problem\cite{firefighter}, where the firefighters are positioned at selected vertices to prevent the outspread of the``fire". Graph burning, on the other hand, aims to minimize the necessary number of steps needed to set all vertices of a graph on ``fire" in a discrete deterministic process. During the burning process, each vertex is in one of two states: \emph{burned} or \emph{unburned}. At the initial round $t=0$, all vertices are unburned. For each subsequent round $t\geq 1$, if an unburned vertex is adjacent to another already burned vertex, then its status turns to burned automatically; otherwise, it can be chosen as the sole candidate to be burned in this round. Once a vertex becomes burned, this status is maintained throughout the remainder of the burning process. The process continues until no vertices remain unburned, and we refer to the graph as \textit{burned}.
	
	Suppose a graph $G$ is burned in $m$ rounds. Let $x_i$ represent the vertex chosen for burning in the $i$-th round, where $1 \leq i \leq m$. We define the sequence ($x_1, x_2,  \dots, x_m$) as the \emph{burning sequence} for the graph $G$, where each $x_i$ is called a \textit{burning source}. The \textit{burning number} of $G$ is defined as the least number of rounds required for all vertices of $G$ to be burned, and it is denoted by $b(G)$. If a graph $G$ can be completely burned in at most $m$ rounds (i.e.~$b(G)\leq m$), then it is said to be \textit{$m$-burnable}. 
	
	Burning number has been extensively studied for many families of graphs. Bonato et al.~\cite{bonato2016how} have shown that the burning numbers of every connected graph would be the smallest among the burning numbers of its spanning trees. They further proved that for all cycles and paths with $n$ vertices is $\left\lceil \sqrt{n}\right\rceil$ and proposed the following conjecture:\\
	
	\begin{burningconjecture} \cite{bonato2016how}
		If $G$ is a connected graph, then $b(G) \leq \left\lceil \sqrt{\vert V(G)\vert}\right\rceil$.
	\end{burningconjecture}

	The burning number conjecture is repeatedly verified for specific families of graphs, for example, theta graphs \cite{liu2019burning},  \textit{t}-unicyclic graphs\cite{zhang2022burning},   circulant graphs \cite{burningcirculant2017}, grid and interval graphs\cite{gupta2021burning}, gear graph and sun graph\cite{geargraph2025},    generalised Petersen graphs \cite{sim2018burning}, caterpillars \cite{liu2020burning,hiller2021burning}, and path forests \cite{liu2021burning,bonato2019bounds,das2018burning}. 
	As for spiders, it garnered the attention of Bonato and Lidbetter \cite{bonato2019bounds} and Das et al.~\cite{das2018burning} to prove independently that these simple trees obey the burning number conjecture. Later, Tan and Teh \cite{tan2020graph} showed that the order of a spider can be pushed plainly beyond $m^2$ while guaranteeing $m$-burnability and a tight upper bound was established. They have also studied the burnability of double spiders and path forests in \cite{tanteh2022doublespider}. More recently, Devroye et al.~\cite{devroye2025burning} showed that random trees have a significantly smaller burning number bound on the order of $n^{1/3}$.
	
	The NP-completeness of the graph burning problem was established for several graph classes such as spiders, path forests, general graphs~\cite{bessy2017burning}, and directed trees~\cite{janssen2020burning}. Meanwhile, improved upper bounds on the burning number in general are well sought after. Initially, for any connected graph $G$ with order $n$,	Bessy et al.~\cite{bessy2018bounds} established an upper bound of $b(G) \leq \sqrt{\frac{12n}{7}}+3$. This upper estimate was subsequently enhanced by Land and Lu~\cite{land2016upper} to $\left\lceil \frac{-3+\sqrt{24n+33}}{4}\right\rceil$. Bastide et al.~\cite{bastide2022improved} recently obtained an improved upper bound of $\left\lceil \frac{8+\sqrt{12n + 64}}{3}\right\rceil$. Almost simultaneously, Norin and Turcotte \cite{burning2022holdasymototically} made a remarkable contribution in establishing that the burning number conjecture holds asymptotically, that is, $b(G)\leq (1+o(1))\sqrt{n}$. Meanwhile, Murakami~\cite{murakami2024conjecture} showed that the burning number conjecture holds for the class of trees that contain no vertices of degree two.
	
	Aside from focusing on the graph burning problem and the burning number conjecture, some researchers work on variations of the burning number. The variation $k$-burning problem was initially investigated by \cite{mondal2021apxhardness}. This variation allows $k$ vertices to be chosen each round where the classical burning is where $k=1$. Later, the $k$-burning problem was shown to be NP-complete for spiders and permutation graphs \cite{GORAIN202383}. Furthermore, the upper bounds and approximation algorithms of the $k$-slow burning problem, a variation that limits the number of neighbours a burning vertex may burn every round, have been studied~\cite{hiller2025upper}. Additionally, graph burning is also studied from the perspectives of parameterized complexity~\cite{kare2019parameterized, kobayashi2022parameterized} and randomness~\cite{mitsche2017burning}. Approximation algorithms for graph burning have also been proposed in \cite{bonato2019approximation}. In 2023, an oriented variation of burning number was studied, where given an undirected graph, the aim is to turn it into a directed graph such that the burning process is prolonged~\cite{courtiel2023orientable}.
	
	In this work, given a pre-determined burning number,	we identify the corresponding extremal trees with the largest attainable order through the introduction of admissible sequences. In Section~\ref{Preliminaries}, we begin with some terminology and then present some general properties of graph burning along with a slight extension of our prior work regarding maximal trees that possess a designated burning number. In the subsequent section, the concept of admissible sequences over a homeomorphically irreducible tree is explained. We show that any tree with a stated burning number whose order is the largest in its homeomorphic class is induced by some admissible sequences (see Theorem \ref{3singleneigbourhood}), and trees induced by admissible sequences with the same characteristic possess the same order (see Lemma~\ref{samesigandorder}). Therefore, comparison among suitable admissible sequences is an effective method to identify the extremal trees. In the following section, we illustrate how this framework involving admissible sequences and the results in Section~\ref{T-admissiblesequence} are applicable to identify the extremal trees with four branch vertices. Finally, we study the smallest diameter attainable by an $n$-spider, when it is required to have a prescribed burning number and its order is the largest among $n$-spiders with equal burning numbers.  

	\section{Preliminaries}\label{Preliminaries}
	In this work, we consider only simple, finite, and undirected graphs. Let $G$ be such a graph. The \textit{distance} between the vertices $u$ and $v$ of $G$, denoted by $d(u, v)$, is the length of the shortest path connecting $u$ and $v$. The \textit{diameter} of $G$ is given by diam($G$)=$\max\{d(v,u)\mid u,v \in V(G)\}$. The number of edges connected to a vertex $v$ is the \textit{degree} of $v$.
	
	Let $(x_1,x_2,x_3,\dots,x_m)$ be a burning sequence of a graph $G$ of length $m$. It is an \textit{optimal} burning sequence if $b(G)=m$. The fire from $x_i$ spreads and, in $m$ rounds, burns every vertex located within distance $m-i$ from it in $G$. Observe that if there is a path of order $2(m-i)+1$ that is a subgraph of $G$, then in a burning process of $m$ rounds, placing the $i$-th burning source at the centre of the path would completely burn the path. For any $k \in \mathbb{N}$, the $k$-th \emph{closed neighbourhood} $N_{k}[v]$ of a vertex $v$ is  the set $\{\,u\in V(G)\mid d(u,v)\leq k\,\}$.	By \cite{bonato2016how}, we know that $(x_1,x_2,\dotsc ,x_m)$ constitutes a burning sequence if and only if $N_{m-1}[x_1] \cup N_{m-2}[x_2] \cup\cdots\cup N_{0}[x_m]=V(G)$ and $d(x_i,x_j)\geq j-i$ whenever $i<j$. Hence, we say that $N_{m-i}[x_i]$ is the \emph{neighbourhood associated} to the burning source $x_i$ and together they form the \emph{associated neighbourhoods} for the given sequence.
	
	Since the burning number of a connected graph is given by the smallest burning number attained across its spanning trees, we focus on trees in this work. Throughout this work, $T$ is a tree. A \textit{leaf} of $T$ is a vertex in $T$ with degree one. A \textit{branch vertex} is a vertex with a degree greater than two. The set of branch vertices of $T$ is denoted as $V_{br}(T)$. An \textit{internal path} refers to a path connecting some two branch vertices without another branch vertex in between, while the path going from a branch vertex to a leaf is called an \textit{arm}. Two trees are defined to be \textit{homeomorphic} provided they are the same up to the extension or contraction of the arms or internal paths by deleting or inserting some vertices. A tree $T$ is \textit{homeomorphically irreducible} if there is no other homeomorphic tree with a smaller order; equivalently, $T$ has no vertices of degree two. A tree is called an $n$-spider if it contains exactly one vertex of degree $n \geq 3$, known as the \textit{head}, with all remaining vertices having a degree at most two.
	
	A rooted tree $T$ is a tree in which a branch vertex is designated as the \textit{root}, denoted as $\rt(T)$. The root serves as the beginning point for traversal. The \textit{predecessor} of $v$ in $T$ refers to the unique vertex adjacent to $v$ from the root leading towards $v$.  The empty rooted tree is denoted as $T_{\emptyset}$. For any vertex $v\in V(T)$, the rooted tree $T[v]$ is the subtree of $T$ with root $v$ and all its descendants in $T$.
	
	Now, we present some earlier known general properties regarding the burning sequence and its associated neighbourhoods of extremal  trees with respect to order with a specified burning number. These results will be the foundation of the framework in Section~\ref{T-admissiblesequence}.  
	
	\begin{theorem}\label{novertexbetween}\cite{LeongSimTeh2023spider}
		Suppose $T$ is a tree and $b(T)=m$. Suppose 
		compared to homeomorphic trees having equal burning numbers, $T$ has the largest order. Then for any burning sequence of $T$ that is optimal, the following properties are satisfied
		for any neighbourhood associated to a burning source:
		\begin{enumerate}
			\item  If some branch vertices are contained in the neighbourhood, then one of the branch vertices is the chosen vertex for the burning source.
			\item If at least two branch vertices are contained in the neighbourhood, then no vertex of $T$ of degree two lies strictly in some path that connects a pair of branch vertices within the neighbourhood.			
		\end{enumerate}		   
	\end{theorem}
	
	Assume $T$ is a tree and $b(T)=m$. Suppose $T$ is no longer $m$-burnable after extending any of its arm or  internal path by inserting a vertex. Then $T$ is said to be \textit{maximally $m$-burnable}. 
	
	\begin{theorem}\label{maximaldisjoint}\cite{LeongSimTeh2023spider}
		Suppose $T$ is a tree such that it is maximally $m$-burnable. Then the associated neighbourhoods are mutually non-overlapping for every burning sequence of $T$ that is optimal.	
	\end{theorem}
	
	Now, we prove the converse of Theorem~\ref{maximaldisjoint} by adding an additional condition.
	
	\begin{theorem}\label{maximalconverse}
		Suppose $T$ is a tree with burning number $m$. Then $T$ is maximally $m$-burnable if the following conditions hold:
		\begin{enumerate}
			\item The associated neighbourhoods  are mutually non-overlapping for every burning sequence of $T$ that is optimal.
			\item All leaves of $T$ are burned in the last round. 			
		\end{enumerate}		
	\end{theorem}
	\begin{proof}
		We argue by contraposition. Assume $T$ is not maximally $m$-burnable. By definition, there exists a $T'$ homeomorphic to $T$ with $V(T')=V(T)+1$ and $b(T')=m$. Consider any burning sequence for $T'$. If there exists some leaf that does not burn last or the associated neighbourhoods are not mutually non-overlapping, then we see that it induces a burning sequence of $T$ with the same property, and we are done. Therefore, we may assume that the associated neighbourhoods  are mutually non-overlapping for every burning sequence of $T'$ that is optimal, and all leaves of $T'$ are burned in the last round.
		
		Suppose $T'$ is obtained from $T$ by adding a vertex at some arm. Therefore, we found a burning sequence of $T$ where some leaves of $T$ are not burned in the last round. Suppose $T'$ is obtained from $T$ by adding a vertex at some internal path. This implies there exists a vertex in $T$ burned by two of the associated neighbourhoods. In other words, the associated neighbourhoods are not mutually non-overlapping.
	\end{proof}
	
	The following corollary is the direct consequence of Theorem~\ref{maximaldisjoint} and Theorem~\ref{maximalconverse}.
	
	\begin{corollary}\label{bothsidemaximaldisjoint}
		Suppose $T$ is a tree and $b(T)=m$. Then $T$ is maximally $m$-burnable if and only if for every optimal burning sequence of $T$, all leaves of $T$ are burned in the last round and the associated neighbourhoods are mutually non-overlapping.
	\end{corollary}

	Path forests are closely related to spider graphs when it pertains to graph burning. In this context, a path forest can be easily obtained from a spider graph by removing the neighbourhood associated with the burning source that  take care of the spider's head in the burning process. The following definition is useful in Section \ref{smallestdiameter} and the existence of $L_n$ was proved in \cite{tanteh2022doublespider}.
	
	\begin{definition}\label{atleastLn}\cite{tanteh2022doublespider}
		Let $n\geq 2$. The number $L_n$ is defined to be the smallest integer satisfying the property: for any path forest $T$ with $n$ independent paths, if the order of $T$ is $m^2$ for some integer $m$ and its shortest path has order at least $L_n$, then its burning number equals $m$.
	\end{definition}	
	
	The following lemma extends Theorem~\ref{maximaldisjoint} with a further property that is useful in Section \ref{T-admissiblesequence}.
	
	\begin{lemma}\label{080623a}
		Let $m>n\geq 1$ and let $T$ be a tree with $n$ branch vertices and burning number $m$. Suppose out of trees homeomorphic to $T$ and having burning number $m$,	$T$ possesses the biggest order among them. Then for every optimal burning sequence $(x_1,x_2,\dotsc, x_m)$ of $T$, 
		there exists $1\leq l\leq  n$ such that 
		$x_1, x_2, \dotsc, x_l$ are all branch vertices and 
		$V_{br}(T) \subseteq \bigcup_{i=1}^l N_{m-i}[x_i]$.
	\end{lemma}

	\begin{proof}
		It suffices to show that for every optimal burning sequence $(x_1,x_2,\dotsc, x_m)$ of $T$, there exists $1\leq l\leq n$ such that 
		$N_{m-i}[x_i]\cap V_{br}(T) \neq \emptyset$ for all $1\leq i\leq l$ and $N_{m-i}[x_i]\cap V_{br}(T) = \emptyset$ for all $l+1\leq i\leq m$ because $x_i$ must then  be a branch vertex for all $1\leq i \leq l$ by Theorem~\ref{novertexbetween}. We proceed by contradiction. Suppose $(x_1,x_2,\dotsc, x_m)$ is some optimal burning sequence satisfying $N_{m-i}[x_i]\cap V_{br}(T) = \emptyset$ and $N_{m-i-1}[x_{i+1}]\cap V_{br}(T) \neq \emptyset$ for some $1\leq i<m$. By Theorem~\ref{novertexbetween}, $x_{i+1}$ must be a branch vertex. By Corollary~\ref{bothsidemaximaldisjoint}, the subtree induced by $N_{m-i}[x_i]$ must be a path of order $2(m-i)+1$. 
		Let $T'$ be a new tree attained from $T$ by contracting the arm containing $x_i$ through the removal of  two vertices belonging to $N_{m-i}[x_i]$ but not $x_i$ and by adding a vertex to each arm or internal path joined to $x_{i+1}$. It can be observed  that $\vert T'\vert \geq \vert T\vert +1$. Furthermore, it can be verified that $(x_1, x_2, \dotsc, x_{i-1}, x_{i+1}, x_i, x_{i+2}, \dotsc, x_m)$ is a burning sequence of the new tree. However, $T'$ is homeomorphic to $T$ and has an order larger than $T$. Hence, this contradicts our hypothesis about $T$.
	\end{proof}

	\section{T-Admissible Sequences and Their Induced Extremal Trees}\label{T-admissiblesequence}

	\begin{definition} \label{subtree}
		Suppose $T$ is a homeomorphically irreducible tree. A finite sequence \mbox{$\langle T_1,T_2,T_3,\dots , T_l \rangle$} of  (possibly empty)   rooted subtrees    of $T$  is said to be a \emph{$T$-admissible sequence} if the collection \mbox{$\{V(T_i) | 1\leq i\leq l \}$} forms a partition of $V_{br}(T)$.
	\end{definition}
	
	\begin{definition}
		Let $T$ be a homeomorphically irreducible tree and suppose \mbox{$t=\langle T_1,T_2,T_3,\dots , T_l \rangle$} is a $T$-admissible sequence. The \textit{signature of $t$} is the function $\sig_{t} : V_{br}(T) \mapsto \mathbb{Z^+}$ defined by: for all $v \in V_{br}(T)$,  
		\begin{equation*}
			\sig_{t}(v)=\dist(v,\rt(T_i))+i,
		\end{equation*}
		where $i$ is the unique integer such that $v \in V(T_i)$. 
	\end{definition}
	
	\begin{definition}\label{inducedatree}
		Let $T$ be a homeomorphically irreducible tree and suppose $t = \langle T_1,T_2,T_3,\dots , T_l \rangle$ is a $T$-admissible sequence. Let $m$ be an integer such that $m>\max\{\sig_t(v)|v\in V_{br}(T) \}$. (Note that $m$ need not be greater than $l$.) A tree $T''$ is said to be \textit{induced by $t$ of degree $m$} if $T''$ is a tree obtained from $T$ by extending the internal paths and arms of $T$ according to the following rules in  two stages.
		\\
		
		Stage 1: Obtain the unique tree $T'$ from $T$ as follows.
		\begin{enumerate}
			\item Extend every arm of $T$ that connects to a branch vertex, say $v$, by $m-\sig_t(v)-1$ vertices.
			\item For any two adjacent branch vertices $v$ and $v'$ of $T$ such that $v\in V(T_i)$ and $v'\in V(T_j)$ for some $i \neq j$, extend the internal path joining $v$ and $v'$ by   $2m-\sig_t(v)-\sig_t(v')$ vertices. 
		\end{enumerate}
		Stage 2: Obtain $T''$ (not unique) from $T'$ as follows.
		\begin{enumerate}
			\item For any $1\leq i\leq m$ such that $T_i=T_{\emptyset}$ or $i\geq l+1$, extend any internal path of $T'$ not of length one or any arm of $T'$ by $2(m-i)+1$ vertices.
			\item The extensions in Stage 2 can be repeatedly applied on the same arm or internal path.
		\end{enumerate}
	\end{definition}
	
	\begin{remark}\label{2104d}
		If $t$ is a  $T$-admissible sequence, then the order of any tree induced by $t$ of degree $m$ is uniquely determined by $t$ and $m$. 
	\end{remark}
	
	\begin{definition} \label{2104b}
		Let $T$ be a homeomorphically irreducible tree. Suppose $s=\langle S_1,S_2,S_3,\dotsc,S_l \rangle$  and	$t=\langle T_1,T_2,T_3,\dots , T_l \rangle$ are $T$-admissible sequences of equal length $l$. We say that $s$ is a \emph{reduction of $t$} if for some integers $1\leq i<j\leq l$, $w\in V(T_i)$, and $v\in V(T_j)$ such that $w$ adjacent to $v$, the following holds:
		\begin{itemize}
			\item $S_k=T_k$ whenever $k\notin \{i,j\}$; 
			\item $\rt(S_i)=\rt(T_i)$;
			\item $\rt(S_j)=\rt(T_j)$ if $v$ is not $\rt(T_j)$;
			\item $V(S_i)=V(T_i)\cup V(T_j[v])$;
			\item $V(S_j)=V(T_j)\backslash V(T_j[v])$;
			\item $\sig_s(v)=\sig_{t}(v)$.
		\end{itemize}
	\end{definition}
	
	\begin{remark}\label{2104c}
		Note that $S_i$ is obtained by ``combining" $T_i$ and $T_j[v]$ by joining $w$ to $v$ with an edge.
	\end{remark}
	
	\begin{lemma}\label{samesigandorder}
		Let $T$ be a homeomorphically irreducible tree. Suppose $s$ and $t$ are $T$-admissible sequences such that $s$ is a reduction of $t$. Then $\sig_s=\sig_{t}$ and the order of any tree induced by $s$ of degree $m$ is equal to that of any tree induced by $t$ of the same degree.
	\end{lemma}
	\begin{proof}
		Suppose $t=\langle T_1,T_2,T_3,\dots , T_l \rangle$ and $s=\langle S_1,S_2,S_3,\dots,S_l\rangle$ are two $T$-admissible sequences such that $s$ is a reduction of $t$. The properties stated in Definition~\ref{2104b} hold for some $1\leq i<j \leq l$, $w\in V(T_i)$, and $v\in V(T_j)$. First, we claim that $sig_t=sig_{s}$. Suppose $z\in V_{br}(T)$. If $z\in V(S_k)$ such that $S_k=T_k$ where $k\notin\{i,j\}$, then $\sig_t(z)=\sig_s(z)$. Thus, we consider only three cases. First, consider the case $z\in V(S_i)$ and $z\notin V(T_i)$. This implies that $z\in V(T_j[v])$. Thus, $\sig_s(z) =\sig_s(v) +\dist(v,z)=\sig_t(v) +\dist(v,z)=\sig_t(z)$. Secondly, when $z\in V(S_i)$ and $z\in V(T_i)$, $\sig_t(z)  = i  + \dist(\rt(T_i),z)= i + \dist(\rt(S_i),z)=\sig_s(z)$ because $\rt(T_i)=\rt(S_i)$. Lastly, consider the case $z\in V(S_j)=V(T_j)\backslash V(T_j[v])$. Since $\rt(S_j)=\rt(T_j)$, similarly, we can conclude that $\sig_s(z)=\sig_t(z)$.
		
		Next, we prove the second part of the lemma. We have shown that the signature of $s$ and $t$ are equal. First, we deal with the case where $v$ is not the root of $T_j$. In this case,  let $w_t$ be the predecessor of $v$ in $T_j$ and let $w_s$ be the predecessor of $v$ in $S_i$. (In fact, $w_s=w$.) Note that $\sig_s(w_t)=\sig_s(w_s)=\sig_s(v)-1=\sig_t(v)-1=\sig_t(w_t)=\sig_t(w_s)$. Since $S_i$, $S_j$, $T_i$, and $T_j$ are all non-empty in this case, by Remark~\ref{2104d}, it suffices to show that the unique tree $S'$ obtained from $T$ after Stage 1 according to $s$ has the same order as the unique tree $T'$ obtained from $T$ after Stage 1 according to $t$.
		
		Since $w_s$ and $v$ are adjacent, $w_s\in V(T_i)$ and $v\in V(T_j)$, it follows that to obtain $T'$, the internal path joining $w_s$ and $v$ is extended by $2m- \sig_t(w_s)-\sig_t(v)$ vertices. On the other hand, since $w_s, v \in V(S_i)$, we did not external the internal path joining $w_s$ and $v$ in obtaining $S'$. Similarly, to obtain $S'$, the internal path joining $v$ and $w_t$ is extended by $2m- \sig_s(v)-\sig_s(w_t)$ vertices while we did not extend the internal path joining $v$ and $w_t$ in obtaining $T'$. Since $\sig_s(w_t)=\sig_t(w_s)$, we know that 
		$2m- \sig_t(w_s)-\sig_t(v) = 2m- \sig_s(v)-\sig_s(w_t)$. 
		Meanwhile, for any other pair of adjacent branch vertices of $T$, it can be verified that 
		both vertices belong to the same $S_k$ if and only if both vertices belong to the same $T_k$. Since $\sig_s=\sig_{t}$, it follows that if any of the remaining arm or internal path of $T$ is extended by the defined number of vertices to obtain $S'$, then it is extended by the same number of vertices to obtain $T'$. Therefore, the orders of $T'$ and $S'$ are equal.
		
		Next, we consider the case where $v$ is the root of $T_j$. Again, let $w_s$ be the predecessor of $v$ in $S_i$. In the case of the sequence $t$, the internal path joining $w_s$ and $v$ is extended by $2m-\sig_t(w_s)-\sig_t(v)$ vertices to obtain $T'$. On the other hand, there is no vertex added to the internal path between $w_s$ and $v$ in $S'$. Thus, the unique trees $S'$ and $T'$ obtained from $T$ induced by both $s$ and $t$, respectively, in Stage 1 have a difference of $2m-\sig_t(w_s)-j$ vertices.
		
		The number of vertices added in getting $S''$ by $s$ is $2(m-j)+1$ vertices more than $t$ in Stage 2 because $S_j=T_\emptyset$ as $v$ is the root of $T_j$. Recall that $w_s$ is the predecessor of $v$ in $S_i$ and thus $\sig_t(w_s)=\sig_s(w_s)=\sig_s(v)-1=\sig_t(v)-1=j-1$. Therefore, the order of $T''$ and $S''$ are equal since $$[2m-\sig_t(w_s)-j]-[2(m-j)+1]=[2m-(j-1)-j]-[2(m-j)+1]=0.$$
		
		This proof is complete. Hence, we conclude that both sequences $t$ and $s$ induce trees of degree $m$ of the same order if $s$ is a reduction of $t$.
	\end{proof}
	
	\begin{definition}\label{canonicaldefinition}
		Let $T$ be a homeomorphically irreducible tree and suppose $t = \langle T_1,T_2,T_3,\dots , T_l \rangle$ is a $T$-admissible sequence.  The sequence $t$ is said to be \textit{canonical} if for all $1\leq i<j\leq l$, no vertex $w\in V(T_i)$ is adjacent to a vertex $v\in V(T_j)$ such that $\sig_t(v)=\sig_t(w)+1$. 
	\end{definition}
	
	\begin{lemma} \label{reductiontocanonical}
		Suppose $T$ is a homeomorphically irreducible tree and $t$ is a $T$-admissible sequence. Then $t$ has no reduction if and only if $t$ is canonical.
	\end{lemma}
	\begin{proof}
		Let $t=\langle T_1,T_2,T_3,\dots , T_l \rangle$ be a $T$-admissible sequence. Suppose the sequence $t$ has no reduction. Assume $t$ is not canonical. By definition, some vertex $w\in V(T_i)$ is adjacent to a vertex $v\in V(T_j)$ such that $\sig_t(v)=\sig_t(w)+1$ for some $i<j$.
		Consider the sequence $s=\langle S_1,S_2,S_3,\dots,S_l\rangle$ such that $S_k=T_k$ for $k\neq \{i,j\}$, $\rt(S_i)=\rt(T_i)$, $\rt(S_j)=\rt(T_j)$ provided $v\neq \rt(T_j)$, $V(S_i)=V(T_i)\cup V(T_j[v])$, and $V(S_j)=V(T_j)\backslash V(T_j[v])$. Since $w\in V(S_i)$ and $\rt(S_i)=\rt(T_i)$, it follows that $\sig_t(v)=\sig_t(w)+1=\sig_s(w)+1=\sig_s(v)$. Clearly, $s$ is a reduction of $t$. Hence, a contradiction occurs.
		
		Conversely, suppose $t$ is a canonical sequence. Assume there exists a reduction $s=\langle S_1,S_2,S_3,\dots,S_l\rangle$ of $t$. This implies that for some integers $1\leq i<j\leq l$, $w\in V(T_i)$, and $v\in V(T_j)$ such that $w$ is adjacent to $v$, all the properties in Definition~\ref{2104b} hold.
		Since $w\in V(T_i)$ and $\rt(T_i)=\rt(S_i)$, by Remark~\ref{2104c}, it follows that $\sig_s(w)=\sig_t(w)$. Hence, $\sig_t(v)=\sig_s(v)=\sig_s(w)+1=\sig_t(w)+1$, which gives a contradiction because $t$ is canonical.
	\end{proof}

	\begin{lemma} \label{canonicalunique}
		Let $T$ be a homeomorphically irreducible tree. Suppose $s$ and $s'$ are $T$-admissible sequences of equal length $l$ with the same signature. If $s$ and $s'$ are both canonical, then $s=s'$.
	\end{lemma}
	
	\begin{proof}
		Suppose $s=\langle S_1,S_2,S_3,\dots,S_l\rangle$ and \mbox{$s'=\langle S'_1,S'_2,S'_3,\dots,S'_l\rangle$} are $T$-admissible sequences with the same signature such that both $s'$ and $s$ are canonical. We argue by contradiction. Assume $s\neq s'$. Say $S_i\neq S_i'$ and $S_k=S'_k$ for all $k<i$ for some $i$. Note that $\rt(S_i)=\rt(S_i')$ or else $\sig_s(\rt(S_i))<\sig_{s'}(\rt(S_i))$. Hence, without loss of generality, there exists $v\in V(S_i)$ such that $v\in V(S_j')$ for some $j>i$. Let $w_s$ be the predecessor of $v$ in $S_i$.
		If $w_s\notin S'_i$ (and thus $w_s\in V(S_j')$ for some $j>i$), then let $w_s$ be the newly chosen $v$. The process of passing to a new $v\in V(S_i)$ must end because the newly chosen $v$ cannot be the root of $S_i$. Hence, we may suppose $w_s\in V(S'_i)$. However, $v\in V(S'_j)$ and $\sig_{s'}(v)=\sig_s(v)=\sig_s(w_s)+1=\sig_{s'}(w_s)+1$. Therefore, it contradicts the cononicalness of the sequence $s'$.
	\end{proof}
	
	\begin{theorem}
		Let $T$ be a homeomorphically irreducible tree. Suppose $s$ and $s'$ are $T$-admissible sequences with the same signature. Then the order of any tree induced by $s$ of a given degree is equal to that of any tree induced by $s'$ of the same degree. 
	\end{theorem}	
	
	\begin{proof}
		First of all, we make a key observation. Suppose $s$ is a $T$-admissible sequence. Let \mbox{$s=s_0, s_1, s_2, \dotsc, s_n$} be a sequence of \mbox{$T$-admissible} sequences such that $s_{i+1}$ is a reduction of $s_i$ for every $i\leq n-1$ and, furthermore, it terminates when $s_n$ has no further reduction. (Such a sequence must terminate, as when going to a reduction, some of the branch vertices are moved to the earlier terms.) By Lemma~\ref{samesigandorder}, the conclusion of the theorem holds for $s$ and $s_n$. Also, $s_n$ is canonical by Lemma~\ref{reductiontocanonical}.
		
		Now, suppose $s$ and $s'$ are $T$-admissible sequences with the same signature. First of all, we may assume $s$ and $s'$ have equal length, as extending any admissible sequence by empty sets will not change its collection of induced trees of a given degree. By our key observation, our proof is complete if canonical $T$-admissible sequences with a given signature and length are unique. This uniqueness is guaranteed by Lemma~\ref{canonicalunique}.	
	\end{proof}
	
	\begin{lemma}\label{080623b}
		Let $T$ be a homeomorphically irreducible tree and suppose $t=\langle T_1, T_2, \dotsc, T_l\rangle$ is a $T$-admissible sequence. Suppose $T''$ is a tree homeomorphic to $T$, where we identify their corresponding branch vertices. Let $m$ be an integer such that $m > \max\{\sig_t(v)\mid v\in V_{br}(T) \}$.
		Then $T''$ is a tree induced by $t$ of degree $m$ if and only if there exists a burning sequence $(x_1,x_2,\dotsc, x_m)$ of $T''$ such that the following holds:
		\begin{enumerate}
			\item the associated neighbourhoods are mutually disjoint;
			\item all leaves of $T''$ are burned in the last round;
			\item for each $1\leq i\leq l$, if $T_i\neq T_\emptyset$ then
			\begin{enumerate}
				\item $x_i= \rt(T_i)$;
				\item  $V(T_i)\subseteq N_{m-i}[x_i]$;
				\item there is no vertex of $T''$ of degree two that lies strictly in some path that connects a pair of distinct branch vertices within the neighbourhood $N_{m-i}[x_i]$.
			\end{enumerate}
		\end{enumerate}
	\end{lemma}

	\begin{proof}
		Suppose	$T''$ is a tree induced by $t$ of degree $m$. 
		According to Definition~\ref{inducedatree}, $T''$ is obtained from $T$ in two stages and so let $T'$ denote the unique intermediate tree obtained in Stage $1$. We define a burning sequence $(x_1,x_2,\dotsc, x_m)$ of $T''$ in two stages as well. 
		For every $1\leq i\leq l$ such that $T_i\neq T_\emptyset$, let $x_i=\rt(T_i)$. First, we show that these burning sources burn all the vertices of $T''$ in $m$ rounds that can be identified with the vertices of $T'$. 
		
		Suppose $v$ and $v'$ are any two adjacent branch vertices of $T$ such that $v\in V(T_i)$ and $v'\in V(T_j)$ for some $i\neq j$. By Definition~\ref{inducedatree}, the internal path joining $v$ and $v'$ is extended by $2m-\sig_t(v)-\sig_t(v')$ to get $T'$ and there is no vertex of degree two that lies strictly in some path that connects a pair of branch vertices belonging to $V(T_i)$ in $T'$ and $T''$. Since $x_i = \rt(T_i)$, when the fire from $x_i$ burns $v$, there are still 
		$(m-i) - \dist(v, \rt(T_i))=  m- \sig_t(v)$
		number of rounds left. (From this observation, it can be deduced that $V(T_i)\subseteq N_{m-i}[x_i]$.) Similarly, when the fire from $x_j$ burns $v'$, there are still $m -\sig_t(v')$ number of rounds left. Hence, together the fires spread from $x_i$ and $x_j$ burns those $2m-\sig_t(v)-\sig_t(v')$ vertices.
		Similarly, for any arm of $T'$ joined to a branch vertex, say $v\in V(T_i)$, it can be shown that the $m-\sig_t(v)$ vertices of the arm are burned by the fire spread from $x_i$.
		
		Now, we consider the vertices added in Stage 2. Be reminded that the extensions are carried out sequentially according to  Definition~\ref{inducedatree}.	
		For each $1\leq i\leq m$ such that $T_i= T_\emptyset$ or $i\geq l+1$, by the definition, $2(m-i)+1$ vertices are added to some arm or internal path, which forms a path that is unburned by the burning sources from Stage 1 or that has been assigned thus far in Stage 2. These new vertices can be burned by choosing $x_i$ to be the middle vertex of the path. From our construction, it can be now be verified that the neighbourhoods associated to the burning sequence are mutually non-overlapping and all leaves of $T''$ are burned in the last round.

		Conversely, suppose there exists a burning sequence $(x_1,x_2,\dotsc, x_m)$ of $T''$ satisfying Properties (1)--(3).
		For each $1\leq i\leq m$ such that  $T_i=T_\emptyset$ or $i\geq l+1$, it can be deduced that the subtree of $T''$ induced by $N_{m-i}[x_i]$ is a path of order $2(m-i)+1$ due to the properties and also because $V(T_i)\cap V_{br}(T)=\emptyset$. Meanwhile, if $T_i\neq T_{\emptyset}$ and $v\in V(T_i)$, then for any arm of $T''$ joined to $v$ or any internal path of $T''$ joining $v$ to some branch vertex outside of $V(T_i)$, it can be deduced from the properties that
		$(m-i) - \dist(v, \rt(T_i))= m-\sig_t(v)$ of its vertices are included in $N_{m-i}[x_i]$. We can now see that $T''$ can be obtained as a tree induced by $t$ of degree $m$, where the vertices of the neighbourhoods that are paths are added exactly in Stage~2.	
	\end{proof}

	\begin{theorem}\label{3singleneigbourhood}
		Let $m>n\geq 1$. Suppose $T$ is a homeomorphically irreducible tree with $n$ branch vertices. 
		\begin{enumerate}
			\item Every tree induced by a $T$-admissible sequence of degree $m$ is $m$-burnable.
			\item Suppose $T_{max}$ has the biggest order when comparing to trees homeomorphic to $T$ that have burning number $m$. Then the following holds.
			\begin{enumerate}
				\item $T_{max}$ is a tree induced by some $T$-admissible sequence of degree $m$.
				\item If $t=\langle T_1, T_2, \dotsc, T_l\rangle$ is a $T$-admissible sequence such that $T_i = T_{\emptyset}$ and $T_{i+1}\neq T_{\emptyset}$ for some $1\leq i<l$, then $T_{max}$ cannot be a tree induced by $t$ of degree $m$.	
				\item If $t=\langle T_1, T_2, \dotsc, T_l\rangle$ is a $T$-admissible sequence such that  $\rt(T_i)$  and $\rt (T_j)$ are adjacent in $T$ for some $1\leq i< j\leq l$ with $j-i\neq 1$, then $T_{max}$ cannot be a tree induced by $t$ of degree $m$.
			\end{enumerate}
		\end{enumerate}
	\end{theorem}

	\begin{proof}
		Part 1 follows immediately from Lemma~\ref{080623b}.
		
		For Part (2)(a), consider any burning sequence $(x_1,x_2,\dotsc, x_m)$ of $T_{max}$ that is optimal. 
		Since $T_{max}$ is maximally $m$-burnable, by Corollary~\ref{bothsidemaximaldisjoint}, the associated neighbourhoods are mutually disjoint
		and all leaves of $T_{max}$ are burned in the last round.
		By Lemma~\ref{080623a}, for some $l\leq  n$, we know that $x_1, x_2, \dotsc, x_l$ are all branch vertices and
		$V_{br}(T)= V_{br}(T_{max}) \subseteq \bigcup_{i=1}^l N_{m-i}[x_i]$.	
		For each $1\leq i\leq l$, let
		$T_i$ be the rooted subtree of $T$ with
		$V(T_i) = V_{br}(T)\cap N_{m-i}[x_i] $  and  $ \rt(T_i)=x_i$.
		Then $t=\langle T_1, T_2, \dotsc, T_l\rangle$ is a $T$-admissible sequence. Furthermore, it can be deduced that
		$m > \max\{\sig_t(v)\mid v\in V_{br}(T) \}$ because $m>n$ and $T_i\neq T_\emptyset$ for all $1\leq i \leq l$.
		By Theorem~\ref{novertexbetween}, for each $1\leq i\leq l$, there is no vertex of degree two that lies strictly in some path connecting a pair of distinct branch vertices within the neighbourhood
		$N_{m-i}[x_i]$. Therefore, by Lemma~\ref{080623b}, it follows that $T_{max}$ is  a tree induced by $t$ of degree $m$.
		
		Now for Part 2(b), assume $T_{max}$ is induced by some $T$-admissible sequence $t=\langle T_1, T_2, \dotsc, T_l\rangle$ of degree $m$ such that $T_i = T_{\emptyset}$ and $T_{i+1}\neq T_{\emptyset}$ for some $1\leq i<l$. By Lemma~\ref{080623b}, it can be deduced that for some burning sequence $(x_1,x_2,\dotsc, x_m)$ of $T_{max}$ that is optimal,  $x_{i+1}$ is a branch vertex while
		$N_{m-i}[x_i]$ contains no branch vertex. However, this contradicts Lemma~\ref{080623a}.	
		
		For Part 2(c), suppose $t=\langle T_1, T_2, \dotsc, T_l\rangle$ is a $T$-admissible sequence such that $\rt(T_i)$ and $\rt(T_j)$ are adjacent in $T$ for some $1\leq i<j\leq l$ with $j>i+1$. Consider the $T$-admissible sequence $s=\langle S_1, S_2, \dotsc, S_l\rangle$ defined by: 
		$$S_k= T_k \text{ for } k\notin \{i,j\}, S_j = T_\emptyset, V(S_i)= V(T_i)\cup V(T_j), \text{and }\rt(S_i)= \rt(T_i).$$
		For every $v\in V_{br}(T)$, if $v\in V(T_j)$ then	
		$\sig_t(v)=\dist(v, \rt(T_j))+j$ while
		$\sig_s(v)	= \dist(v, \rt(S_i))+i=
		\dist(v, \rt(T_j))+1+i$ and thus $\sig_t(v) = \sig_s(v)+(j-i-1)$; otherwise, it can be verified that $\sig_s(v)= \sig_t(v)$. Note that $\rt(T_i)$ is the only vertex in $T_i$ adjacent to some vertex in $T_j$. 
		
		Consider the unique trees $S'$ and $T'$ obtained from $T$ induced by $s$ and $t$, respectively, of degree $m$ in Stage 1. In getting $T'$, the internal path in $T$ joining $\rt(T_i)$ and $\rt(T_j)$ is extended by $2m-i-j$ vertices while the same internal path is not extended in getting $S'$. On the other hand,
		for any arm of $T$ joined to a vertex in $T_j$
		or any internal path of $T$ joining a vertex in $T_j$ and a branch vertex outside $V(T_i)\cup V(T_j)$, it is extended by an extra $j-i-1$ vertices in getting $S'$ compared to getting $T'$
		because $\sig_t(v) = \sig_s(v)+(j-i-1)$ for any $v\in V(T_j)$. There are at least two such arms or internal paths combined, and this can be deduced by considering any leaf of $T_j$ (or $\rt(T_j)$ if $T_j$ has only one vertex). Meanwhile, the effect on any other arm or internal path of $T$ is the same in  getting $S'$ and $T'$.	
		Therefore, $\vert V(T')\vert - \vert V(S')\vert$ is at most $(2m-i-j)-2(j-i-1)$.
		
		However, any tree induced by $s$  gets an extra $2(m-j)+1$ vertices in Stage 2 compared to that induced by $t$ because $S_j=T_\emptyset$. Therefore, since
		$$[2(m-j)+1]- [(2m-i-j)-2(j-i-1)]= j-i-1>0,$$
		it follows that the order of any tree induced by $s$ is strictly larger than that of any tree induced by $t$ of the same degree. By Part 1, any tree induced by $s$ of degree $m$ is $m$-burnable. Hence, $T_{max}$ cannot be a tree induced by $t$ of degree $m$ for otherwise, we would have a larger homeomorphic tree with burning number $m$.
	\end{proof}

	\section{Case Study: Four Branch Vertices} \label{application}
	By Theorem~\ref{3singleneigbourhood} (Part 2a), if a tree $T'$ is \emph{extremal} in the sense that it has the biggest order among homemorphic trees that share equal burning numbers, then $T'$ is induced by a $T$-admissible sequence where $T$ is the homeomorphically irreducible tree homeomorphic to $T'$. Therefore, identifying such trees $T'$ with burning number $m$ by comparing the corresponding canonical $T$-admissible sequences would be a good strategy. By Theorem~\ref{3singleneigbourhood}, some of these sequences can be eliminated from our consideration. In this section, we will illustrate how our developed idea in Section~\ref{T-admissiblesequence} can be used to identify extremal trees with four branch vertices bearing a prescribed burning number through canonical $T$-admissible sequences. As an emphasis, throughout this and the next section, by calling a tree extremal means that it has the biggest order among trees that are sharing the same burning number and homeomorphic to the given tree.
	
	A tree with four branch vertices can be in the form of $A-B-C-D$ or \chemfig{A-B(-[2]D)(-C)} where $``-"$ represents an internal path or an edge between any two branch vertices. We let $\deg(A)=a$, $\deg(B)=b$, $\deg(C)=c$, and $\deg(D)=d$. In this section, for simplicity, we adopt some notation to represent a rooted tree. For instance, we write $t=\langle A_{BC},D \rangle$ for a canonical $T$-admissible sequence with $V(T_1)=\{A,B,C\}$ where $A$ is the root of $T_1$ and $V(T_2)= \{D\}$ where $D$ is the root of $T_2$.
	
	
	First, we consider a tree in the form of $A-B-C-D$ and let $t$ be a canonical $T$-admissible sequence of such tree. Suppose $\rt(T_1)=A$. Since $t$ is canonical, $B$ is not the root of $T_2$. Also, by Theorem~\ref{3singleneigbourhood} (Part 2c), $B$ is neither the root of $T_3$ nor $T_4$. Hence, under the condition that $t$ is canonical, we only consider the following possibilities of $t$: $\langle A_{BCD}\rangle$, $\langle A,C_{BD}\rangle$, $\langle A_B,C_D\rangle$, $\langle A,D_{CB}\rangle$, and $\langle A_{BC},D\rangle$. However, $\langle A_{BCD}\rangle$ is a reduction of $\langle A_{BC},T_\emptyset,T_\emptyset,D\rangle$, $\langle A,C_{BD}\rangle$ is a reduction of $\langle A,C_D,B\rangle$, and  $\langle A,D_{CB}\rangle$ is a reduction of $\langle A,D_C,T_\emptyset,B\rangle$. By Lemma~\ref{samesigandorder} and Theorem~\ref{3singleneigbourhood} (Part 2b and 2c), they cannot induce extremal trees with burning number $m$. Therefore, when $\rt(T_1)=A$, the only two possible $t$ that need to be taken into consideration are $\langle A_B,C_D\rangle$ and $\langle A_{BC},D\rangle$.
	
	Now, suppose $\rt(T_1) = B$. Since $t$ is canonical, neither $A$ nor $C$ can be the root of $T_2$. According to Theorem~\ref{3singleneigbourhood} (Part 2c), the root of $T_i$ for $i \in \{3,4\}$ cannot be $A$ or $C$. Also, if $T_2$ is nonempty, then the root of $T_2$ must be $D$. 
	Therefore, the possible canonical $T$-admissible sequences are $\langle B_{ACD}\rangle$, $\langle B_{AC},D\rangle$ and $\langle B_A,D_C\rangle$. Since $\langle B_{ACD}\rangle$ is a reduction of $\langle B_{AC},T_\emptyset,D\rangle$ and $\langle B_A,D_C\rangle$ is a reduction of $\langle B_A,D,C\rangle$, we conclude that $t$ can only be $\langle B_{AC},D\rangle$. Similarly, if $\rt(T_1)=C$ or $\rt(T_1)=D$, then $t$ can only be $\langle C_{DB},A\rangle$, $\langle D_C,B_A\rangle$, and $\langle D_{CB},A\rangle$.
	
	Therefore, we can see that there are a total of six canonical $T$-admissible sequences that can possibly induce an extremal tree with a pre-determined burning number, namely: $\langle A_B,C_D\rangle$, $\langle A_{BC},D\rangle$, $\langle B_{AC},D\rangle$, $\langle C_{DB},A\rangle$, $\langle D_C,B_A\rangle$, and $\langle D_{CB},A\rangle$. 
	
	Now, we compare the number of vertices for any tree induced by each of these six canonical \mbox{$T$-admissible sequences}. Note that all these six possible sequences are of the same length. Following this, the number of vertices added in Stage 2 according to Definition~\ref{inducedatree} for all these six sequences are equal. Thus, the difference in the order of trees induced by each of these six sequences depends on Stage 1 only. Table~\ref{orderinducedbyinstage1} shows the number of vertices added in Stage 1 to get the tree induced by each of the six canonical $T$-admissible sequences. Meanwhile, Table \ref{conditionondegreeforcanonicalsequence} shows the difference in the number of vertices added by the corresponding pair of canonical $T$-admissible sequences in Stage 1. If all the entries in a row are non-negative, then this implies that the corresponding canonical $T$-admissible sequence induces a tree with a larger or equal number of vertices compared to any tree induced by any of the other five canonical $T$-admissible sequences. 
	\begin{table}[htb]
		\centering
		\captionof{table}{The number of vertices added in Stage 1 for the case $A-B-C-D$}\label{orderinducedbyinstage1}
		\scalebox{0.95}{
			\begin{tabular}{|M{3cm}|M{13.5cm}|}
				\hline
				Canonical $T$-admissible sequence & The number of vertices added in getting the tree induced by the canonical $T$-admissible sequence of degree $m$ in Stage 1\\ \hline
				$\langle A_B,C_D\rangle$ &  $(a-1)(m-2)+(b-2+c-2)(m-3)+(d-1)(m-4)+2m-4$\\ \hline
				$\langle A_{BC},D\rangle$ & $(a-1)(m-2)+(b-2+d-1)(m-3)+(c-2)(m-4)+2m-5$ \\ \hline
				$\langle B_{AC},D\rangle$&$(b-2)(m-2)+(a-1+c-2+d-1)(m-3)+2m-4$ \\ \hline
				$\langle C_{BD},A\rangle$ & $(c-2)(m-2)+(a-1+b-2+d-1)(m-3)+2m-4$ \\ \hline
				$\langle D_C,B_A\rangle$ & $(d-1)(m-2)+(b-2+c-2)(m-3)+(a-1)(m-4)+2m-4$ \\ \hline
				$\langle D_{CB},A\rangle$ & $(d-1)(m-2)+(a-1+c-2)(m-3)+(b-2)(m-4)+2m-5$ \\ \hline
			\end{tabular}
		}
		
	\end{table}
	
	\begin{table}[htb]
		\centering
		\captionof{table}{The difference in the numbers of vertices added by the pair of canonical sequences}\label{conditionondegreeforcanonicalsequence}
		\scalebox{0.85}{
			\begin{tabular}{|M{2.5cm}|M{2.5cm}|M{2.5cm}|M{2.5cm}|M{2.5cm}|M{2.5cm}|M{2.5cm}|}
				\hline
				Canonical $T$-admissible sequence & $\langle A_B,C_D\rangle$ & $\langle A_{BC},D\rangle$& $\langle B_{AC},D\rangle$& $\langle C_{BD},A\rangle$& $\langle D_C,B_A\rangle$& $\langle D_{CB},A\rangle$\\ \hline
				$\langle A_B,C_D\rangle$ &0&$c-d$&$a-b-d+2$&$a-c-d+2$&$2a-2d$&$a+b-2d$ \\ \hline
				$\langle A_{BC},D\rangle$ & $d-c$ & 0 & $a-b-c+2$ & $a-2c+2$ & $2a-c-d$ & $a+b-c-d$ \\ \hline
				$\langle B_{AC},D\rangle$ &$b+d-a-2$&$b+c-a-2$&0&$b-c$&$a+b-d-2$&$2b-d-2$ \\ \hline
				$\langle C_{BD},A\rangle$ &$c+d-a-2$&$2c-a-2$&$c-b$&0&$a+c-d-2$&$b+c-d-2$\\ \hline
				$\langle D_C,B_A\rangle$ &$2d-2a$&$c+d-2a$&$d-a-b+2$&$d-a-c+2$&0&$b-a$\\ \hline
				$\langle D_{CB},A\rangle$ &$2d-a-b$&$c+d-a-b$&$d-2b+2$&$d-b-c+2$&$a-b$&0 \\ \hline
			\end{tabular}
		}
	\end{table}
	The second column of Table~\ref{tablefor4branchverticesABCD} presents a list of conditions on the degrees of the branch vertices that exhaust all possibilities for the case of $A-B-C-D$. Under each of these sets of conditions, our analysis shows that the corresponding $T$-admissible sequence in the third column in Table~\ref{tablefor4branchverticesABCD} is the only sequence from Table~\ref{conditionondegreeforcanonicalsequence} such that all entries in its row are nonnegative. Hence, the sequence induces extremal trees with burning number $m$.
	
	\begin{center}
		\begin{table}[htb]
			\setlength\extrarowheight{5pt}
			\captionof{table}{Extremal trees for the case $A-B-C-D$}\label{tablefor4branchverticesABCD}
			\scalebox{0.85}{
				\begin{tabular}{ |M{2.5cm}|M{10cm}|M{4.5cm}| } 
					\hline
					Form of tree with 4 branch vertices& Conditions on the degrees of the branch vertices & Canonical $T$-admissible sequence $t$ that induces extremal trees with burning number $m$\\ \hline
					\multirow{12}{*}{A-B-C-D}
					&$b\geq \max\{a,c,d\}$&$\langle B_{AC},D\rangle$\\ \cline{2-3}
					&$a>b\geq \max\{c,d\}$ and $b+\min\{c,d\}\geq a+2$&$\langle B_{AC},D\rangle$\\ \cline{2-3}
					&$a>b\geq c\geq d$ and $a+2\geq b+d$&$\langle A_B,C_D\rangle$\\ \cline{2-3}
					&$a>b\geq d\geq c$ and $a+2\geq b+c$&$\langle A_{BC},D\rangle$\\ \cline{2-3}
					&$a> c\geq \max\{b,d\}$ and $c+d\geq a+2$&$\langle C_{BD},A\rangle$\\ \cline{2-3}
					&$a>c\geq \max\{b,d\}$ and $a+2\geq c+d$&$\langle A_B,C_D\rangle$\\ \cline{2-3}
					&$a\geq d\geq b \geq c$ and $b+c\geq a+2$&$\langle B_{AC},D\rangle$\\ \cline{2-3}
					&$a\geq d\geq b \geq c$ and $a+2\geq b+c$&$\langle A_{BC},D\rangle$\\ \cline{2-3} 
					
					&$a\geq d\geq c \geq b$ and $2c\geq a+2$ and $ b+c \geq d +2 $&$\langle C_{BD},A\rangle$\\ \cline{2-3}
					&$a\geq d\geq c \geq b$ and $2c\geq a+2$ and $ d+2 \geq b+c $&$\langle D_{CB},A\rangle$\\ \cline{2-3}
					&$a\geq d\geq c \geq b$ and $a+2 \geq 2c$ and $c+d \geq a+b $&$\langle D_{CB},A \rangle$\\ \cline{2-3}
					&$a\geq d\geq c \geq b$ and $a+2 \geq 2c$ and $a+b \geq c+d $&$\langle A_{BC},D
					\rangle$\\
					\hline			
				\end{tabular}
			}
		\end{table}
	\end{center}
	\newpage
	Now, we apply the a similar idea for the case \chemfig{A-B(-[2]D)(-C)}. Let $t$ be a canonical $T$-admissible sequence for a homeomorphically irreducible tree $T$ in the form of \chemfig{A-B(-[2]D)(-C)} where $t$ induces extremal trees with a prescribed burning number. Here, we may assume that $a\geq c \geq d$ for $T$ of this form. We claim that, given a burning number, there are only two possible canonical $T$-admissible sequences that can induce extremal trees (see Table~\ref{tablefor4branchverticesTshape}). For the case $\rt(T_1)=B$, since $t$ is canonical, $A$, $C$ and $D$ are not the roots of $T_2$. Thus, the only possible canonical $T$-admissible sequence is $\langle B_{ACD}\rangle$.
	
	For the case $\rt(T_1)=A$, similarly, $B$ is not the root of $T_2$. Also, by Theorem~\ref{3singleneigbourhood} (Part 2c), $B$ is neither the root of $T_3$ nor $T_4$. Therefore, we have nine possible canonical sequences: $\langle A_{BCD}\rangle$, $\langle A_{BD},C\rangle$, $\langle A,C_B,D\rangle$, $\langle A,C,D_B\rangle$, $\langle A,C_{BD}\rangle$, $\langle A_{BC},D\rangle$, $\langle A,D_B,C\rangle$, $\langle A,D,C_B\rangle$, and $\langle A,D_{BC}\rangle$ .
	
	Note that $\langle A_{BCD}\rangle$ is a reduction of $\langle A_{BC},T_\emptyset,D\rangle$, $\langle A,C,D_B\rangle$ is a reduction of $\langle A,C,D,B\rangle$, $\langle A,C_{BD}\rangle$ is a reduction of $\langle A,C,B_D\rangle$, $\langle A,D,C_B\rangle$ is a reduction of $\langle A,D,C,B\rangle$, and $\langle A,D_{BC}\rangle$ is a reduction of $\langle A,D,B_C\rangle$. Hence, they cannot induce trees with a prescribed burning number that are extremal according to Lemma~\ref{samesigandorder} and Theorem~\ref{3singleneigbourhood} (Part 2b and 2c). Therefore, if $\rt(T_1)=A$, then $t$ can only be $\langle A_{BD},C\rangle$, $\langle A,C_B,D\rangle$, $\langle A_{BC},D\rangle$, and $\langle A,D_B,C\rangle$. 
	
	By using a similar argument, if $\rt(T_1)=C$,  then $t$ can only be $\langle C_{BD},A\rangle$, $\langle C,A_B,D\rangle$, $\langle C_{BA},D\rangle$, $\langle C,D_B,A\rangle$. If $\rt(T_1)=D$, then $t$ can only be $\langle D_{BA},C\rangle$, $\langle D,C_B,A\rangle$, $\langle D_{BC},A\rangle$, and $\langle D,A_B,C\rangle$. 
	
	Now, we claim that $\langle A_{BD},C\rangle$ induces a tree with a larger order than that of $\langle A_{BC},D\rangle$, $\langle C_{BD},A\rangle$, $\langle C_{BA},D\rangle$, $\langle D_{BA},C\rangle$ and $\langle D_{BC},A\rangle$. Note that these six sequences are all of the same length. Hence, it is sufficient to just consider the vertices added in Stage 1 as in Definition \ref{inducedatree}. The vertices added in Stage 1 by $\langle A_{BD},C\rangle$ is $(a-1)(m-2)+(b-3+c-1)(m-3)+(d-1)(m-4)+(2m-4)$. Note that the only internal path being extended in Stage 1 by these six sequences is of equal length and consists of $2m-4$ vertices. This is because the extended internal path is always between $B$ and $\rt(T_2)$ and $\sig_t(B)=\sig_t(\rt(T_2))=2$ for $t$ being any of the six sequences. Therefore, the difference in the number of vertices added by $\langle A_{BD},C\rangle$ and the rest of the five sequences depends on the coefficients of $(m-2), (m-3)$ and $(m-4)$. With the assumption $a\geq c\geq d$, this implies that $\langle A_{BD},C\rangle$ induces the largest order among these six sequences. 
	
	Using a similar argument, we claim that $\langle A,C_B,D\rangle$  induces a larger order of tree than that of  $\langle A,D_B,C\rangle$, $\langle C,A_B,D\rangle$, $\langle C,D_B,A\rangle$, $\langle D,C_B,A\rangle$, and $\langle D,A_B,C\rangle$.
	
	Therefore, for a homeomorphically irreducible tree $T$ in the form of \chemfig{A-B(-[2]D)(-C)} and with the assumption $a\geq c \geq d$, to obtain trees that are extremal with a specified burning number, if suffices to take into consideration one of three possible canonical $T$-admissible sequences, namely:  $\langle B_{ACD}\rangle$, $\langle A_{BD}, C\rangle$, and $\langle A, C_B, D\rangle$.
	
	Note that these canonical sequences are of different lengths. Table~\ref{verticesintshape} shows the total number of vertices added in getting any tree induced by each of the canonical $T$-admissible sequences of degree $m$ in both stages as in Definition~\ref{inducedatree}. Table~\ref{difference2} shows the difference in the number of vertices between pairs of canonical $T$-admissible sequences in Stage 1. A row of non-negative entries indicates that the corresponding canonical $T$-admissible sequence induces a tree with at least as many vertices as any other canonical $T$-admissible sequence. Thus, using Table~\ref{difference2}, we could identify the canonical $T$-admissible sequence that induces trees that are extremal with a pre-determined burning number, as shown in Table~\ref{tablefor4branchverticesTshape}.

	\begin{table}[htb] 
		\centering
		\setlength\extrarowheight{4pt}
		\captionof{table}{The number of vertices added in both Stage 1 and 2 for the case \protect\chemfig{A-B(-[2]D)(-C)}} \label{verticesintshape}
		\scalebox{0.85}{
			\begin{tabular}{|M{3cm}|M{13.5cm}|}
				\hline
				Canonical $T$-admissible sequence & The number of vertices added in getting any tree induced by the canonical $T$-admissible sequence of degree $m$\\ \hline
				$\langle B_{ACD}\rangle$
				& $(b-3)(m-2)+(a-1+c-1+d-1)(m-3)+(m-1)^2$ \\ \hline
				$\langle A_{BD},C\rangle$
				& $(a-1)(m-2)+(b-3+c-1)(m-3)+(d-1)(m-4)+(2m-4)+(m-2)^2$ \\  \hline
				$\langle A,C_B,D\rangle$
				& $(a-1)(m-2)+(c-1)(m-3)+(b-3+d-1)(m-4)+(2m-4)+(2m-6)+(m-3)^2$ \\ \hline
			\end{tabular}
		}
	\end{table}

	\begin{table}[htb]
		\centering
		\setlength\extrarowheight{4pt}
		\captionof{table}{The difference in the numbers of vertices added by the pair of canonical sequences} \label{difference2}
		\scalebox{0.85}{
			\begin{tabular}{|M{4cm}|M{2.8cm}|M{2.8cm}|M{2.8cm}|}
				\hline
				Canonical $T$-admissible sequence & $\langle B_{ACD}\rangle$ & $\langle A_{BD},C\rangle$ & $\langle A,C_B,D\rangle$\\ \hline
				
				$\langle B_{ACD}\rangle$ &0&$-a+b+d-2$&$-a+2b+d-3$\\ \hline
				$\langle A_{BD},C\rangle$&$a-b-d+2$&0&$b-2$\\ \hline
				$\langle A,C_B,D\rangle$&$a-2b-d+3$&$2-b$&0\\ \hline
				
			\end{tabular}
		}
	\end{table}

	
	\begin{center}
		\begin{table}[htb]
			\setlength\extrarowheight{5pt}
			\captionof{table}{Extremal trees for the case \protect\chemfig{A-B(-[2]D)(-C)} }\label{tablefor4branchverticesTshape}
			\scalebox{0.85}{
				\begin{tabular}{ |M{2.5cm}|M{10cm}|M{4.5cm}| } 
					\hline				
					Form of tree with 4 branch vertices& Conditions on the degrees of the branch vertices & Canonical $T$-admissible sequence $t$ that induces extremal trees with burning number $m$\\ \hline
					\multirow{2}{*}{\chemfig{A-B(-[2]D)(-C)}}
					&$a\geq c\geq d$ and $b+d\geq a+2$&$\langle B_{ACD}\rangle$\\ \cline{2-3}
					&$a\geq c\geq d$ and $ a+2\geq b+d$&$\langle A_{BD},C\rangle$\\
					\hline
				\end{tabular}
			}
		\end{table}
	\end{center}

	\section{Results on the Smallest Diameter} \label{smallestdiameter}
	If a graph is connected, its burning number is bounded by $r+1$, where $r$ is the radius of the graph. This is obviously true when considering a burning sequence with the initial burning source positioned at the centre of the graph. In this section, we study the smallest possible diameter of an extremal $n$-spider with a prescribed burning number. (As a reminder, by being extremal, it means the $n$-spider has the biggest order among $n$-spiders that share equal burning numbers.)
	
	\begin{lemma}\label{differatmostone}
		Suppose $T$ is an $n$-spider such that the length of any two arms differ by at most one. Then $T$ has the smallest diameter among all $n$-spiders having the same order.
	\end{lemma}
	
	\begin{proof}
		Assume the conclusion is false. Then there exists an $n$-spider $T'$ with $|T'|=|T|$ and $\diam(T)>\diam(T')$. Let the arm lengths of $T$ (respectively, $T'$) be $l_1\geq l_2 \geq \dots \geq l_n$ (respectively, $l'_1\geq l'_2 \geq \dots \geq l'_n$). By our assumption, $l_1+l_2>l'_1+l'_2$. Let $l_1+l_2-(l'_1+l'_2)=k$ for some $k\in\mathbb{Z^+}$. By the hypothesis, $l_i\geq l_1-1$ for all $3\leq i \leq n$. Since $|T'|=|T|$, it follows that $$\sum_{i=3}^{n}l'_i=\sum_{i=3}^{n}l_i+(l_1+l_2)-(l'_1+l'_2)\geq (n-2)(l_1-1)+k\geq (n-2)(l_1-1)+1.$$
		
		This implies that $l'_3\geq l_1$ or else $\sum_{i=3}^{n}l'_i\leq (n-2)(l_1-1)$. However, this means that $l'_1+l'_2\geq 2l'_3\geq 2l_1\geq l_1+l_2$, which gives a contradiction.
	\end{proof}
	
	Before diving into the main result of this section, let us made an observation about the properties of extremal $n$-spiders that possess a given burning number. This observation can be verified by Theorem~\ref{novertexbetween} or our findings in Section~\ref{T-admissiblesequence}.
	
	Observation: Let $m\geq 2$ and $n\geq 3$. An $n$-spider $T'$ is extremal with burning number $m$ if and only if $T'$ has a burning sequence of length $m$ such that the neighbourhoods associated to the sequence are mutually non-overlapping, the head is chosen as where to put the first burning source, and all leaves are burned in the last round.
	
	\begin{theorem}
		Let $n\geq 3$. Assume $T$ is an extremal $n$-spider with burning number $m$, and moreover it has the smallest diameter among such $n$-spiders. 
		\begin{enumerate}
			\item If $3\leq m\leq 2n-1$, then the diameter of $T$ is $6m-10$.
			\item  If $m$ belongs to one of the following cases, then $T$ is the unique $n$-spider (up to isomorphism) of order $n(m-1)+1+(m-1)^2$ such that the length of any two arms differ by at most one:
			\begin{enumerate}
				\item $m=2nq+i$ for some $q\geq 1$ and $i\in\{0,1,2\}$;
				\item $\left\lfloor\frac{(m-1)^2}{n}\right\rfloor \geq L_n$.
			\end{enumerate}
		\end{enumerate}
	\end{theorem}
	\begin{proof}
		(Proof of Part 1) Let $3\leq m\leq 2n-1$. Suppose $T$ is an extremal $n$-spider with $b(T)=m$. By the observation, it follows that  
		$$\diam(T)\geq (2m-1)+ (2m-3)+ (2m-5)-1= 6m-10,$$
		regardless of whether the second and the third burning sources are put at the same arm or not.
		
		Now, we will show that the lower bound $6m-10$ can be attained. 
		Suppose $n+2\leq m\leq 2n-1$. 
		Let $t= m- n-1$ and note that $m = 2t+ (n-t)+1$.
		Consider the $n$-spider $T$ where the first  $t$ arms have equal length of $(m-1)+4t$ and the lengths of the remaining $n-t$ arms (at least two) are $$(m-1)+(4t+1), (m-1)+ (4t+3), \dotsc, (m-1)+(2m-3).$$   
		Consider a burning sequence where the first burning source put at the head burns $m-1$ vertices from each arm in $m$ rounds. The remaining unburned vertices by the first burning source from the first $t$ arms are burned using the last $2t$ burning sources  because 
		$$4t= (4t-1) +1= (4t-3)+3= \dotsb = (2t+1)+ (2t-1).$$
		while the leftover burning sources take care of the 
		ones from the remaining arms.
		Clearly, the associated neighbourhoods are mutually disjoint and all leaves are burned in the last round. 	
		Therefore, by the observation, $T$ is an extremal $n$-spider, its burning number equals $m$, and its diameter is $6m-10$.
		
		The case $3\leq m \leq n+1$ can be argued similarly with a different but simpler $n$-spider. Hence, its proof is omitted.
		
		\bigskip
		
		(Proof of Part 2a) Consider the case $m=2nq+2$ for some fixed $q\geq 1$.	
		Let $t_{i}=2(m-i-1)+1$ for $1\leq i \leq 2nq$. It can be verified that
		$$\sum_{k=1}^{q}(t_{2n(k-1)+1}+t_{2nk})=\sum_{k=1}^{q}(t_{2n(k-1)+2}+t_{2nk-1})=\dots= \sum_{k=1}^{q}(t_{2n(k-1)+n}+t_{2nk+1-n})=L$$
		$$\text{and }\bigcup_{j=1}^n \bigcup_{k=1}^q \{ t_{2n(k-1)+j}, t_{2nk+1-j}    \}=\{t_1, t_2, t_3, \dotsc, t_{2nk}\}= \{3,5,7, \dotsc, 2m-3\}.$$
		Let $L=2mq$ and $T$ be the $n$-spider with all arm lengths equal to  $m-1+L$ except one that is $m-1+L+1$. Note that $|T|=n(m-1)+1+(nL+1)=n(m-1)+1+(m-1)^2$.
		
		By burning the head in the first round, in a burning process of $m$ rounds, consider the leftover vertices unburned by the initial burning source. They form a path forest of order $(m-1)^2$ with path orders $L,L, \dotsc, L, L+1$. Due to the above equalities, this path forest is $(m-1)$-burnable and it follows that there is a burning sequence of $T$ of length $m$ with the property that the associated neighbourhoods are mutually non-overlapping, and all leaves are burned in the last round. Therefore, by the observation, $T$ is an extremal $n$-spider among  $n$-spiders burning numbers of which are $m$.
		By Lemma~\ref{differatmostone}, $T$ is the unique $n$-spider with the smallest diameter among $n$-spiders with the same order. 	
		
		Similar arguments apply to the cases $m=2nq$ and $m=2nq+1$, and thus, the proof of this part is complete.
		
		\bigskip
		
		(Proof of Part 2b) Consider the $n$-spider $T$ of order $n(m-1)+1+(m-1)^2$ such that the length of any two arms differ by at most one. Its arm length is either $(m-1)+\left(\left\lfloor\frac{(m-1)^2}{n}\right\rfloor\right)$ or one larger. Similarly, consider a burning strategy of $T$ where we burn the head in the first round. The leftover vertices unburned by the initial burning source in $m$ rounds make a path forest of order $(m-1)^2$ with each path having order at least  $\left\lfloor\frac{(m-1)^2}{n}\right\rfloor\geq L_n$. Hence, by Definition~\ref{atleastLn}, this path forest is $(m-1)$-burnable. The rest of the argument can be completed similarly as in Part~2(a). 
	\end{proof}

	\section{Statements and Declarations}
	
	On behalf of all authors, the corresponding author states that there is no conflict of interest. Data sharing is not applicable---no data was used/generated for this work.
	The first author and the corresponding author acknowledge the financial support by the Research University Grant awarded to Wen Chean Teh by Universiti Sains Malaysia with grant number 1001/PMATHS/8011129. This work is a continuation of our work published in \cite{LeongSimTeh2023spider} and is part of the thesis submitted for the fullfillment of the requirement for the first author's Master of Science degree. To our best knowledge, the framework and results here are definitely new. The preprint of this work was posted on arXiv as arXiv:2504.20427v2.

\end{document}